\documentclass[11pt, oneside]{amsart}
\usepackage[all,cmtip]{xy}
\usepackage[utf8x]{inputenc}
\usepackage{amsmath, amsthm, amssymb}
\usepackage{eucal, mathrsfs, yfonts}
\usepackage{amsfonts}
\usepackage{array}
\usepackage{hyperref}
\usepackage{graphicx}
\usepackage{caption}
\usepackage{pdfpages}
\usepackage{verbatim}
\usepackage{float}
\usepackage[english]{babel}
\usepackage{epstopdf}
\usepackage{tikz}
\usetikzlibrary{calc}
\usepackage{pgfplots}
\usepackage[section]{placeins}
\usepackage{geometry}                		
\newtheorem{thm}{Theorem}[section]
\theoremstyle{plain}

\newtheorem{prop}[thm]{Proposition}

\theoremstyle{remark}

\numberwithin{equation}{section}

\newtheorem*{remark}{Remark}

\begin{document}

\title{Non-degeneracy of the harmonic structure \\on Sierpi\'nski Gaskets}

\author[Konstantinos Tsougkas]{Konstantinos Tsougkas}

\address{Konstantinos Tsougkas\\
        Department of Mathematics\\
        Uppsala university, Sweden}

\date{\today}

\email{konstantinos.tsougkas@math.uu.se}

\maketitle

\begin{abstract}

We prove that the harmonic extension matrices for the two dimensional level-$k$ Sierpi\'nski gasket are invertible for every $k \geq 2$. This has been previously conjectured to be true by Hino in \cite{hino2009energy} and \cite{hino2015some} and tested numerically for $k \leqslant 50$. We also give a necessary condition for the non-degeneracy of the harmonic structure for general finitely ramified self-similar sets based on the vertex connectivity of their first graph approximation.

\end{abstract}

\section{Introduction}\noindent
The Dirichlet problem for the Laplace operator has been studied in a variety of settings: domains, manifolds, graphs. One newer context is that of analysis on fractals \cite{ben1999not,kigami2001analysis,MR1761364,MR1765920,MR1761365}. However harmonic functions on fractals exhibit a notable difference compared to those of $\mathbb{R}^d$. Among many properties of harmonic functions on $\mathbb{R}^d$, it is known (see for example \cite{MR1805196}) that if a harmonic function defined on a domain $\Omega$ is constant on a non-empty open subset of $\Omega$, then it is constant everywhere in $\Omega$. However this does not hold in the case of fractals where we can have examples of non-constant harmonic functions being constant on smaller cells, in which case we say that we have a \emph{degenerate harmonic structure}. Such examples include the Snowflake set, the Vicsek set and the Hexagasket constructed from three boundary vertices \cite{strichartz2006differential} among others. A widely studied self-similar set, often being used as a prototype for most results in the theory, is the two dimensional Sierpi\'nski gasket. A variant of it can be created by dividing the line segments of the initial triangle into $k \geq 2$ segments of equal length which gives us a family of self-similar sets called the Sierpi\'nski gaskets of level $k$, with the familiar Sierpi\'nski gasket denoted as $SG_2$. The non-degeneracy of the harmonic structure is well known for $SG_2$ and $SG_3$ and Hino has checked it numerically for all $k \leqslant 50$ and has conjectured it to be the case for all $SG_k$ in \cite{hino2009energy} and \cite{hino2015some}. The aim of this paper is to give an affirmative answer to this conjecture and also give a necessary condition for a self-similar fractal to have a non-degenerate harmonic structure. Our main theorem is the following.

\begin{thm}
For every $k \geq 2$ the harmonic structure on the two dimensional $SG_k$ is non-degenerate.
\end{thm}

We will now give specific details regarding the topic. The Sierpi\'nski gaskets of level $k$ are the attractor of the iterated function system $F_i(x)=x/k  +b_{i,k}$
for some proper choice of $b_{i,k}$, in which case $SG_k$ is the unique non-empty compact set such that
$$SG_k=\bigcup_{i=1}^{\frac{1}{2}k(k+1)} F_i(SG_k).$$

\begin{figure}
\centering
\includegraphics[scale=0.41]{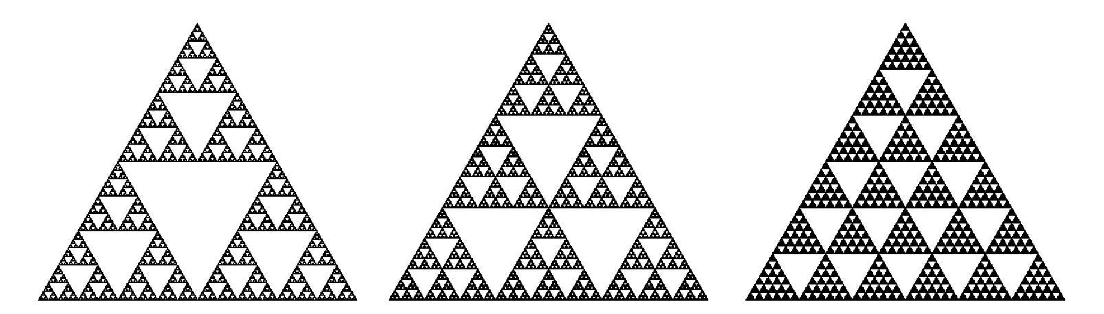}
\caption{Sierpi\'nski gaskets of level $2$, $3$, and $6$.}
\end{figure}

These are post critically finite self-similar sets and their boundary is always defined to be the set of vertices of the outermost triangle and is denoted by $V_0=\{q_1,q_2,q_3\}$. A set of the form $F_iK$ for some $i\in \{1,2,\dots ,k(k+1)/2\}$ is called a \emph{cell}.
These self-similar sets can be approximated via a sequence of so-called fractal graphs with $G_0$ being the complete graph on the boundary and then $G_m$ being $(k(k+1)/2)^m$ copies of it identified at appropriate points.  
On these graphs we can define the \emph{renormalized energy} of a function as
$${\mathcal E}_m(u,v)=r^{-m} \sum_{x\sim_m y} (u(x)-u(y))(v(x)-v(y))$$
where $r$ is called the renormalization constant which for example equals $3/5$ for $SG_2$ but is different for each $SG_k$. The exact value of $r$ for each $SG_k$ is unknown, some investigations regarding that have been made at \cite{freiberg2013exact} and \cite{hambly2002asymptotics}. By taking the limit we have the energy form which is an inner product on the space of functions of finite energy modulo constants. Harmonic functions on $G_m$ are functions with fixed values at $V_0$ and the rest of the values chosen so that they minimize the energy of the graph. Alternatively, they are characterized by solving the Dirichlet problem
$$\Delta_m h(x)=0 \text{ for every } x \notin V_0$$ where $\Delta_m$ is the discrete graph Laplacian. In that case, the values of a harmonic function can be determined on $G_1$ by solving a system of linear equations. On any given cell, we have for $1 \leqslant i \leqslant k(k+1)/2$ that
$$\left(
\begin{array}{ccc}
h(F_iq_1)\\
h(F_iq_2)\\
h(F_iq_3)
\end{array}
\right)=A_i \left(
\begin{array}{ccc}
h(q_1)\\
h(q_2)\\
h(q_3)
\end{array}
\right)$$
where the matrices $A_i$ are called \emph{harmonic extension matrices}. If they are invertible for every $i$ the harmonic structure is called non-degenerate. A harmonic extension matrix $A_i$ being singular is easily seen to be equivalent to the existence of a non-constant harmonic function being constant on the cell $F_iK$. There is also a probabilistic interpretation connecting random walks on graphs with harmonic functions and electrical networks, we refer the reader to \cite{snell2000random} for a detailed exposition. We denote by $V_m$ the vertex set of $G_m$ and $V^{*}=\cup_{n=0}^{\infty} V_n$. Then $V^{*}$ is dense in $SG_k$ and since functions of finite energy are always uniformly continuous it suffices to study them on $V^{*}$. In the case that the harmonic structure is non-degenerate, we have that the space of harmonic functions is $3$-dimensional with a basis being $h_i(q_j)=\delta_{ij}$ for $i,j=1,2,3$. By prescribing the values at the boundary, we can inductively evaluate the values of the harmonic function for each $G_m$ and thus in the limit for $V^{*}$. On $SG_2$ this gives us the familiar ``$1/5-2/5$" extension rule. The Sierpi\'nski gaskets can also be constructed in higher dimensions where the conjecture regarding the non-degeneracy of their harmonic structure at the time of this writing remains open. In this paper we will only focus on the two dimensional case.

The Laplace operator is then defined weakly via integration against a measure, the most common choice being the Hausdorff measure with a proper normalization. However attention has been given recently to so-called energy measures which are defined through the energy of functions on the fractal graph approximations. For a function $f$ of finite energy, its \emph{energy measure} $\nu_f$ is defined as $$\nu_{f}(F_wK)= \lim_{m\to \infty}r^{-m}\sum_{\{x,y\in F_wV_0; \; x\sim_m y\}} \left(f(x)-f(y)\right)^2.$$ 
If we pick an orthonormal basis, with respect to the energy inner product, of harmonic functions modulo constants we can define the \emph{Kusuoka measure} as $\nu=\nu_{h_1}+\nu_{h_2}$ and this definition is independent of the choice of the orthonormal basis. Moreover the Kusuoka measure is singular with respect to the Hausdorff one and it can also be shown that every energy measure is absolutely continuous with respect to the Kusuoka measure. Further information regarding the Kusuoka measure can be found among others at \cite{bell2012energy,johansson2015ergodic,kusuoka1989dirichlet}. A motivating factor for our investigation has been the expansion of known results to a wider class of self-similar sets. Particularly, in \cite{hino2009energy} and \cite{hino2015some} properties of the energy measures on p.c.f. self-similar sets have been studied in a more general setting and those results hold only for self-similar sets under the assumption that their harmonic structure is non-degenerate. Therefore, an implication of our theorem is that all the results proven under the non-degeneracy assumption are therefore now valid for all two dimensional $SG_k$ and some of their variants as explained in section 2. One such example is a conjecture of Bell, Ho and Strichartz found in \cite{bell2012energy} which was proved by Hino in \cite{hino2015some} in a more general setting for self-similar sets satisfying the non-degeneracy of the harmonic structure assumption. Another example are the results of \cite{pelander2008products}. We can now apply to all $SG_k$ the result found in \cite{hino2009energy} giving us the following characterization of energy measures.
\begin{thm}
For every non-constant harmonic function on $SG_k$ the energy measure $\nu_h$ is a minimal energy-dominant measure. In particular, for any two non-constant harmonic functions $h_1$, $h_2$, the energy measures $\nu_{h_1}$ and $\nu_{h_2}$ are mutually absolutely continuous.
\end{thm}

\section{Barycentric embedding of $SG_k$}
Our approach is based on geometric graph theory, an exposition of which can be found in \cite{lovasz1999geometric}. Recall that a finite undirected graph is called simple if it has no loops or multiple edges, it is called planar if it can be embedded in the plane in a way that its edges never intersect except at their corresponding vertices and it is called $k$-connected if it can not be made disconnected by removing any $k-1$ vertices. Its vertex connectivity is $k$ is the maximal integer $k$ so that the graph is $k$-vertex connected. Moreover, if we have points $x_1,x_2 , \dots , x_k$ in $\mathbb{R}^2$ we call their \emph{barycenter} or \emph{centroid} the point 
$\tilde{x}=\frac{1}{k}\sum_{i=1}^k x_i.$
If we take a simple $3$-connected planar graph, and then place the vertices bounding a face of it on the plane forming a convex polygon, then we call the \emph{rubber band representation} of it the graph created by letting all the other free vertices be positioned at the barycenter of their neighbors. The edges are drawn as straight line segments connecting the proper vertices.  The terminology is motivated by thinking of the edges as rubber bands satisfying Hooke's Law. Tutte's spring theorem, first proven in \cite{tutte1963draw}, states that this algorithm gives us a crossing free plane embedding and moreover that every face of the corresponding planar embedding is convex. This is also known as a Tutte embedding. We give Tutte's spring theorem as stated in \cite{lovasz1999geometric}.
\begin{thm}
Let $G$ be a simple $3$-connected planar graph. Then its rubber band representation gives us an embedding of $G$ into $\mathbb{R}^2$.
\end{thm}

In Figure 2 we present an example of barycentric embedding in the plane of the first graph approximation of $SG_2$ and $SG_3$.
\begin{figure}
\centering
\begin{tikzpicture}[scale=0.37]
\draw (-5,0)--(-1,{2*sqrt(3)});
\draw (5,0)--(1,{2*sqrt(3)});
\draw (5,0)--(0,{sqrt(3)});
\draw (-5,0)--(0,{sqrt(3)});
\draw (-1,{2*sqrt(3)})--(1,{2*sqrt(3)});
\draw (-1,{2*sqrt(3)})--(0,{sqrt(3)});
\draw (1,{2*sqrt(3)})--(0,{sqrt(3)});
\draw (-1,{2*sqrt(3)})--(0,{5*sqrt(3)});
\draw (1,{2*sqrt(3)})--(0,{5*sqrt(3)});
\end{tikzpicture}
\begin{tikzpicture}[scale=0.37]
\draw (-5,0)--(-4/3,{sqrt(3)});
\draw (5,0)--(4/3,{sqrt(3)});
\draw (5,0)--(5/3,{4*sqrt(3)/3});
\draw (-5,0)--(-5/3,{4*sqrt(3)/3});
\draw (-4/3,{sqrt(3)})--(4/3,{sqrt(3)});
\draw (-4/3,{sqrt(3)})--(-5/3,{4*sqrt(3)/3});
\draw (4/3,{sqrt(3)})--(5/3,{4*sqrt(3)/3});
\draw ({-4/3},{sqrt(3)})--(0,{5*sqrt(3)/3});
\draw (4/3,{sqrt(3)})--(0,{5*sqrt(3)/3});
\draw (-5/3,{4*sqrt(3)/3})--(0,{5*sqrt(3)/3});
\draw (5/3,{4*sqrt(3)/3})--(0,{5*sqrt(3)/3});
\draw (-1/3,{8*sqrt(3)/3})--(0,{5*sqrt(3)/3});
\draw (1/3,{8*sqrt(3)/3})--(0,{5*sqrt(3)/3});
\draw (-5/3,{4*sqrt(3)/3})--(-1/3,{8*sqrt(3)/3});
\draw (5/3,{4*sqrt(3)/3})--(1/3,{8*sqrt(3)/3});
\draw (-1/3,{8*sqrt(3)/3})--(1/3,{8*sqrt(3)/3});
\draw (-1/3,{8*sqrt(3)/3})--(1/3,{8*sqrt(3)/3});
\draw (-1/3,{8*sqrt(3)/3})--(0,{15*sqrt(3)/3});
\draw (1/3,{8*sqrt(3)/3})--(0,{15*sqrt(3)/3});
\end{tikzpicture}
\caption{A barycentric embedding of $G_1$ for $SG_2$ and $SG_3$ with boundary vertices fixed at equal distances, positioned at $(0, \sqrt{3}),(-1,0),(1,0)$.}
\end{figure}
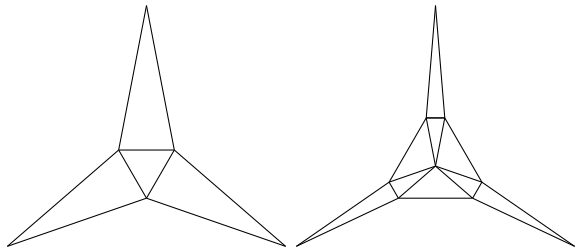
We are now ready to present the proof of Theorem 1.1.
\begin{proof}[Proof of Theorem 1.1]
Let $k \geq 2$ and assume that the harmonic structure on $SG_k$ is degenerate and let $A_i$ be a singular harmonic extension matrix. We will only concern ourselves with the first level graph approximation of $SG_k$ since this is where the harmonic extension matrices are constructed on. The matrix being singular implies that there exists a harmonic function $h$ that is non-constant on the boundary $q_1,q_2,q_3$ but is constant on the cell $F_iK$ with $h(F_i(q_1))=h(F_i(q_2))=h(F_i(q_3))$. By the addition of constants and normalizing we can in fact assume that $h(q_1)=\alpha \geq 1$, $h(q_2)= 1$ and $h(q_3)=0$ with some possible relabeling of the boundary vertices. 

Call $\tilde{G_1}$ the slightly modified $G_1$ graph by adding three extra edges connecting the three boundary vertices $q_1, q_2, q_3$ as in Figure 3.
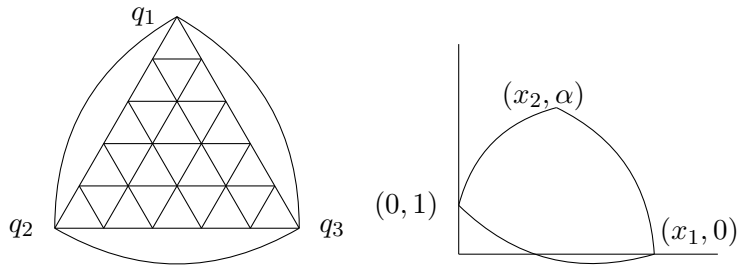
\begin{figure}
\centering
\newcommand*\rows{5}
\hspace*{-1.3cm}
\begin{tikzpicture}[scale = 0.65]
    \foreach \row in {0, 1, ...,\rows} {
        \draw ($\row*(0.5, {0.5*sqrt(3)})$) -- ($(\rows,0)+\row*(-0.5, {0.5*sqrt(3)})$);
        \draw ($\row*(1, 0)$) -- ($(\rows/2,{\rows/2*sqrt(3)})+\row*(0.5,{-0.5*sqrt(3)})$);
        \draw ($\row*(1, 0)$) -- ($(0,0)+\row*(0.5,{0.5*sqrt(3)})$);
    }
    \draw (0,0) to[bend right] (5,0);
    \draw (0,0) to[bend left] (2.5,{2.5*sqrt(3)});
    \draw (2.5,{2.5*sqrt(3)}) to[bend left] (5,0);
    \draw (2.5,{2.5*sqrt(3)}) node (1) [label=left:$q_1$] {};
    \draw (0,0) node (2) [label=left:$q_2$] {};
    \draw (5,0) node (3) [label=right:$q_3$] {};
\end{tikzpicture}
\begin{tikzpicture}[scale=0.65]
\draw (0,0)--(5.3,0);  
\draw (0,0)--(0,4.3);  
\draw (0,1) node (1) [label=left:{$(0,1)$}] {};
\draw (3.7,0.45) node (1) [label=right:{$(x_1,0)$}] {};
\draw (3,3.25) node (1) [label=left:{$(x_2,\alpha)$}] {};

\draw (0,1) to[bend right] (4,0);
\draw (2,3) to[bend left] (4,0);
\draw (0,1) to[bend left] (2,3);
\end{tikzpicture}
\caption{The modified $\tilde{G_1}$, and its embedding in the plane before the application of Tutte's algorithm.}
\end{figure}
Then $\tilde{G_1}$ is obviously a simple planar graph which is $3$-connected since if we take any two vertices of it we can easily find $3$ vertex independent paths connecting them. Moreover, the Dirichlet problem is exactly identical to that of $G_1$ as the Laplace equation need not hold at the boundary vertices. Then we draw $\tilde{G_1}$ in $\mathbb{R}^2$ as shown in Figure 3 in the following way. Put at position $(0,1)$ the vertex $q_2$, then put the vertex $q_3$ at $(x_1,0)$ for some $x_1>0$ and finally the vertex $q_1$ at $(x_2,\alpha)$ for some $x_2>0$. The three vertices bounding the outer face are lying on a triangle and thus satisfy the conditions of Tutte's spring theorem. Applying the theorem gives us that there exists a crossing-free plane embedding such that the position of every interior vertex is the barycenter of the positions of its neighbors. However, the coordinates of each vertex are calculated component-wise and therefore each coordinate function is harmonic at the free non-boundary vertices. In particular, by construction, the $y$ coordinate of all the vertices is exactly the solution of the Dirichlet problem of $SG_k$ with boundary values $h(q_1)=\alpha$, $h(q_2)= 1$ and $h(q_3)=0$. By our assumption, at the cell $F_iK$ the solution of the Dirichlet problem is constant, meaning that the three vertices of that cell in the barycentric embedding have all the same $y$ coordinate, and thus the edges connecting them must overlap, giving us a degenerate face of the graph. But then this is a degenerate embedding contradicting Tutte's theorem.
\end{proof}

\begin{remark}
By Steinitz's theorem we have that that the modified graphs $\tilde{G}$ are the edge graphs of convex $3$-dimensional polyhedra. Moreover, even a planar graph that is not $3$-connected can be barycentrically embedded as we can add extra edges to make it $3$-connected and then remove them after the barycentric embedding. The key element in our proof is that by adding those three extra edges we do not perturb the harmonic structure.
\end{remark}

\begin{remark}

A different proof using the topological properties of the two dimensional $SG_k$ was later given in \cite{cao2017topological}.

\end{remark}

So far in the electric network interpretation we have assumed that every edge of the graph has equal conductance which gives us the \emph{canonical harmonic structure}. However, we can put arbitrary positive conductance in which case renormalization and the existence of a harmonic structure becomes in general a difficult problem on fractals. In our case, Tutte's theorem also holds in the case that the barycenter of the Tutte embedding is not exact, but instead is a convex combination of the other points. This allows us to generalize our result in cases without assuming equal conductance. This technique in fact proves the non-degeneracy of the harmonic structure not just for the Sierpi\'nski gaskets but for other self-similar sets as well as long as we can get a planar simple $3$-connected graph by connecting the boundary vertices in $G_1$ to create the outer face of the graph. We present some examples in Figure 4. The last of these examples is a modified version of the so called Hanoi attractor, in which case the non-degeneracy of the harmonic structure applies only on the triangular cells. This approach fails on the Hexagasket and the Vicsek set as their equivalent modified graphs lack connectivity conditions. This in fact is not a coincidence, a high enough vertex connectivity is required for a non-degenerate harmonic structure. We have the following.

\begin{figure}
\centering

\begin{tikzpicture}[scale=0.35]

\draw (-5,0)--(5,0);
\draw (5,0)--(0,{5*sqrt(3)});
\draw (-5,0)--(0,{5*sqrt(3)});

\draw (0,0)--(5/2,{5/2*sqrt(3)});
\draw (0,0)--(-5/2,{5/2*sqrt(3)});
\draw (-25/8,{15/8*sqrt(3)})--(25/8,{15/8*sqrt(3)});

\draw (-30/8,{10/8*sqrt(3)})--(-10/8,{10/8*sqrt(3)});
\draw (-5/2,0)--(-10/8,{10/8*sqrt(3)});
\draw (-5/2,0)--(-30/8,{10/8*sqrt(3)});

\draw (-5/4,0)--(-15/8,{5/8*sqrt(3)});
\draw (-5/4,0)--(-5/8,{5/8*sqrt(3)});
\draw (-15/4,0)--(-35/8,{5/8*sqrt(3)});
\draw (-15/4,0)--(-25/8,{5/8*sqrt(3)});
\draw (-5/2,{10/8*sqrt(3)})--(-15/8,{15/8*sqrt(3)});
\draw (-5/2,{10/8*sqrt(3)})--(-25/8,{15/8*sqrt(3)});
\draw (-35/8,{5/8*sqrt(3)})--(-25/8,{5/8*sqrt(3)});
\draw (-15/8,{5/8*sqrt(3)})--(-5/8,{5/8*sqrt(3)});

\draw (5/4,0)--(25/8,{15/8*sqrt(3)});
\draw (10/4,0)--(30/8,{10/8*sqrt(3)});
\draw (15/4,0)--(35/8,{5/8*sqrt(3)});
\draw (15/4,0)--(10/8,{20/8*sqrt(3)});
\draw (10/4,0)--(10/8,{10/8*sqrt(3)});
\draw (5/4,0)--(5/8,{5/8*sqrt(3)});

\draw (5/8,{5/8*sqrt(3)})--(35/8,{5/8*sqrt(3)});
\draw (10/8,{10/8*sqrt(3)})--(30/8,{10/8*sqrt(3)});

\draw (5/8,{15/8*sqrt(3)})--(-10/8,{30/8*sqrt(3)});
\draw (5/8,{15/8*sqrt(3)})--(15/8,{25/8*sqrt(3)});
\draw (10/8,{20/8*sqrt(3)})--(20/8,{20/8*sqrt(3)});
\draw (-15/8,{25/8*sqrt(3)})--(15/8,{25/8*sqrt(3)});

\draw (-5/8,{15/8*sqrt(3)})--(-15/8,{25/8*sqrt(3)});
\draw (-5/8,{15/8*sqrt(3)})--(-10/8,{15/8*sqrt(3)});
\draw (-15/8,{15/8*sqrt(3)})--(-5/8,{25/8*sqrt(3)});
\draw (-5/8,{15/8*sqrt(3)})--(0,{20/8*sqrt(3)});
\draw (0,{20/8*sqrt(3)})--(-20/8,{20/8*sqrt(3)});

\draw (-10/8,{30/8*sqrt(3)})--(10/8,{30/8*sqrt(3)});
\draw (-5/8,{35/8*sqrt(3)})--(5/8,{35/8*sqrt(3)});

\draw (-5/8,{25/8*sqrt(3)})--(5/8,{35/8*sqrt(3)});
\draw (5/8,{25/8*sqrt(3)})--(-5/8,{35/8*sqrt(3)});
\draw (5/8,{25/8*sqrt(3)})--(10/8,{30/8*sqrt(3)});

\end{tikzpicture}
\begin{tikzpicture}[scale=0.35]

\draw (-5,0)--(5,0);
\draw (5,0)--(0,{5*sqrt(3)});
\draw (-5,0)--(0,{5*sqrt(3)});

\draw (-1,0)--(2,{3*sqrt(3)});
\draw (1,0)--(-2,{3*sqrt(3)});

\draw (-1,0)--(-3,{2*sqrt(3)});
\draw (1,0)--(3,{2*sqrt(3)});
\draw (-3,{2*sqrt(3)})--(3,{2*sqrt(3)});

\draw (-2,{3*sqrt(3)})--(2,{3*sqrt(3)});

\end{tikzpicture}
\begin{tikzpicture}[scale=0.35]

\draw (-5,0)--(5,0);
\draw (5,0)--(0,{5*sqrt(3)});
\draw (-5,0)--(0,{5*sqrt(3)});

\draw (-14/4,0)--(-17/4,{(6/8)*sqrt(3)});
\draw (14/4,0)--(17/4,{(6/8)*sqrt(3)});

\draw (-3/4,0)--(0,{(6/8)*sqrt(3)});
\draw (3/4,0)--(0,{(6/8)*sqrt(3)});

\draw (11/8,{(17/8)*sqrt(3)})--(23/8,{(17/8)*sqrt(3)});
\draw (0,{(6/8)*sqrt(3)})--(11/8,{(17/8)*sqrt(3)});

\draw (-11/8,{(17/8)*sqrt(3)})--(-23/8,{(17/8)*sqrt(3)});
\draw (0,{(6/8)*sqrt(3)})--(-11/8,{(17/8)*sqrt(3)});
\draw (-11/8,{(17/8)*sqrt(3)})--(11/8,{(17/8)*sqrt(3)});

\draw (11/8,{(17/8)*sqrt(3)})--(17/8,{(23/8)*sqrt(3)});
\draw (-11/8,{(17/8)*sqrt(3)})--(-17/8,{(23/8)*sqrt(3)});

\draw (-6/8,{(34/8)*sqrt(3)})--(6/8,{(34/8)*sqrt(3)});

\end{tikzpicture}

\caption{Examples of p.c.f. fractals with non-degenerate harmonic structure.}
\end{figure}
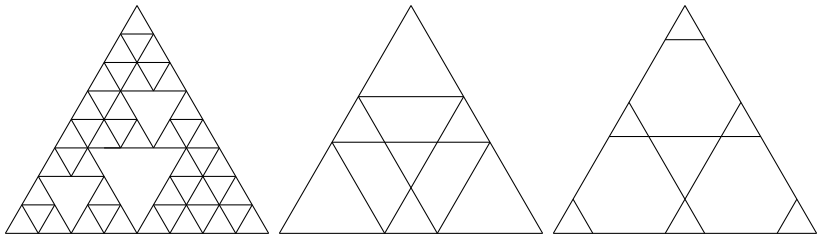

\begin{prop}
Let $\tilde{G_1}$ be the first graph approximation of a finitely ramified self-similar set $K$ with extra edges connecting all the boundary points with each other. If $\tilde{G_1}$ has vertex connectivity less than $|V_0|$, then the harmonic structure of $K$ is degenerate.
\end{prop}

\begin{proof}
Let us assume $\tilde{G_1}$ has vertex connectivity $k$ where $k<|V_0|$. Remove vertices $v_1, \dots, v_k$ from $\tilde{G_1}$ so that the new graph becomes disconnected with at least two connected components. Let $v$ be a vertex of $\tilde{G_1}\setminus\{v_1,\dots,v_k\}$ that lies in a connected component $C$ which is not connected to any boundary point. Such a vertex and connected component will always exist as any boundary points remaining are connected to each other. Now, excluding the extra boundary edges, $\tilde{G_1}$ is by construction made of copies of the complete graph $G_0$. Let us consider a cell $B$, in other words a copy of $G_0$, that $v$ belongs to in $\tilde{G_1}$. Then the vertices of that cell must necessarily be included in the vertices of $C\cup\{v_1,\dots,v_k\}$. Let the $k\times |V_0| $ matrix $P=(p_{ij})$ with $p_{ij}$ being the probability that a simple symmetric random walk on $\tilde{G_1}$ starting at the vertex $v_i$ will arrive first at the boundary vertex $q_j$ before any other boundary vertices. This is in fact the harmonic extension matrix for the vertices $v_1,\dots v_k$. Consider the map
$$f:\mathbb{R}^{|V_0|\times 1} \to \mathbb{R}^{k\times 1}, \hspace{0.5cm} f(u)=Pu.$$
By looking at the corresponding dimensions this map has at least a one dimensional kernel and therefore we can have a non-constant harmonic function on $\tilde{G_1}$ such that it is constant on $v_1,\dots v_k$. But then such a harmonic function will be constant on $C\cup\{v_1\dots,v_k\}$ and thus on the cell $B$ giving us a degenerate harmonic structure.

\end{proof}

In \cite{cao2017topological} the authors observed that for finitely ramified self-similar sets with non-junction inner points the harmonic structure always seems to be degenerate. The above proposition shows that this is indeed always the case because removing the neighboring $|V_0|-1$ points will make that point isolated, thereby the graph connectivity can be at most $|V_0|-1$. Of course this also raises the natural question of whether the necessary condition of at least $|V_0|$-connectedness is also sufficient for non-degeneracy, and if not, then what extra conditions are required. If $|V_0|=2$ then a simple symmetry argument shows that $2$-connectedness and planarity do not suffice for non-degeneracy. But as we have already seen, if $|V_0|=3$, then $3$-connectedness and planarity of $\tilde{G_1}$ with the boundary being on the outer face suffices. Finding such a condition if it exists would settle the conjecture for the higher dimensional Sierpi\'nski gaskets of level $k$ because of their high connectivity as well as give an easy characterization of which harmonic structures are degenerate. The situation seems more delicate than the two dimensional case. However, by results of \cite{linial1988rubber} and a similar argument we can say that it is indeed true for almost all edge weights $c_{ij}$ chosen randomly, independently and uniformly on $[0,1]$, but not necessarily for the canonical harmonic structure.

Based on this approach we can think of the Kusuoka measure as a renormalized energy of the cell of the representation in $\mathbb{R}^2$. We can also give a geometric way of visualizing the Kusuoka measure on these Sierpi\'nski gaskets. In order to visualize $\nu(F_wK)$ where $|w|=m$, we can do the following. For example, in the case of $SG_2$, fix the boundary vertices $V_0$ in the plane at positions $(1/\sqrt{6},1/\sqrt{2})$, $(-1/\sqrt{6},1/\sqrt{2})$ and at the origin $(0,0)$, and perform the barycentric embedding for the $G_m$ graph. We observe that at both coordinates we will have independently the solution to the Dirichlet problem of the system of the orthonormal harmonic functions modulo constants in the definition of the Kusuoka measure. Thus if we define $L_i$ to be the the length of each side of the triangle with vertices $F_w(V_0)$ in the barycentric embedding we get that 
$$\left(\frac{5}{3}\right)^m\sum_{i=0}^2 {L_i}^2= \left(\frac{5}{3}\right)^m\left(\sum_{i\sim j}(x_i-x_j)^2+\sum_{i\sim j}(y_i-y_j)^2\right)=\nu(F_wK)$$
where those sums extend over the three sides of the cell. The Kusuoka measure of a cell $F_wK$ can therefore be visualized as $(5/3)^m$ times the sum of the areas of three squares with side lengths equal to those of the triangles in the barycentric embedding of $G_m$. Alternatively, we can think of it as the renormalized energy of the rubber band representation of the graph in $\mathbb{R}^2$ with appropriate boundary vertex positions. We present in Figure 5 an example of this barycentric embedding for $G_2$ of $SG_2$.

\begin{figure}[!htb]
\centering
\begin{tikzpicture}[scale=4.5]
\draw[line width=0.1pt,gray!30,step=1mm]
(-0.75,0) grid (0.75,0.8);

\draw (0,0) node (1) [label=left:{$(0,0)$}] {};
\draw (-{sqrt(6)/6},{(sqrt(2)/2)}) node (2) [label=left:{$(-\frac{1}{\sqrt{6}},\frac{1}{\sqrt{2}})$}] {};
\draw ({sqrt(6)/6},{(sqrt(2)/2)}) node (2) [label=right:{$(\frac{1}{\sqrt{6}},\frac{1}{\sqrt{2}})$}] {};

\draw (-{1/25*sqrt(6)/6},{9/25*sqrt(2)/2})--(0,0);
\draw ({(1/25)*(sqrt(6))/6},{9/25*sqrt(2)/2})--(0,0);
\draw (-{(1/25)*sqrt(6)/6},{(9/25)*(sqrt(2)/2)})--({(1/25)*sqrt(6)/6},{(9/25)*(sqrt(2)/2)});

\draw (-{(5/25)*sqrt(6)/6},{(15/25)*(sqrt(2)/2)})--(0,{(12/25)*(sqrt(2)/2)});
\draw ({(5/25)*sqrt(6)/6},{(15/25)*(sqrt(2)/2)})--(0,{(12/25)*(sqrt(2)/2)});
\draw (-{(5/25)*sqrt(6)/6},{(15/25)*(sqrt(2)/2)})--({(-1/25)*sqrt(6)/6},{(9/25)*(sqrt(2)/2)});
\draw ({(5/25)*sqrt(6)/6},{(15/25)*(sqrt(2)/2)})--({(1/25)*sqrt(6)/6},{(9/25)*(sqrt(2)/2)});
\draw (-{(1/25)*sqrt(6)/6},{(9/25)*(sqrt(2)/2)})--(0,{(12/25)*(sqrt(2)/2)});
\draw ({(1/25)*sqrt(6)/6},{(9/25)*(sqrt(2)/2)})--(0,{(12/25)*(sqrt(2)/2)});

\draw ({(-5/25)*sqrt(6)/6},{(15/25)*(sqrt(2)/2)})--({(-12/25)*(sqrt(6)/6)},{(20/25)*(sqrt(2)/2)});
\draw ({(-5/25)*sqrt(6)/6},{(15/25)*(sqrt(2)/2)})--({(-7/25)*(sqrt(6)/6)},{(19/25)*(sqrt(2)/2)});
\draw ({(-12/25)*sqrt(6)/6},{(20/25)*(sqrt(2)/2)})--({(-7/25)*(sqrt(6)/6)},{(19/25)*(sqrt(2)/2)});

\draw ({(5/25)*sqrt(6)/6},{(15/25)*(sqrt(2)/2)})--({(12/25)*(sqrt(6)/6)},{(20/25)*(sqrt(2)/2)});
\draw ({(5/25)*sqrt(6)/6},{(15/25)*(sqrt(2)/2)})--({(7/25)*(sqrt(6)/6)},{(19/25)*(sqrt(2)/2)});
\draw ({(12/25)*sqrt(6)/6},{(20/25)*(sqrt(2)/2)})--({(7/25)*(sqrt(6)/6)},{(19/25)*(sqrt(2)/2)});

\draw ({(-12/25)*sqrt(6)/6},{(20/25)*(sqrt(2)/2)})--({(-25/25)*(sqrt(6)/6)},{(25/25)*(sqrt(2)/2)});
\draw ({(-12/25)*sqrt(6)/6},{(20/25)*(sqrt(2)/2)})--({(-11/25)*(sqrt(6)/6)},{(21/25)*(sqrt(2)/2)});
\draw ({(-11/25)*sqrt(6)/6},{(21/25)*(sqrt(2)/2)})--({(-25/25)*(sqrt(6)/6)},{(25/25)*(sqrt(2)/2)});

\draw ({(12/25)*sqrt(6)/6},{(20/25)*(sqrt(2)/2)})--({(25/25)*(sqrt(6)/6)},{(25/25)*(sqrt(2)/2)});
\draw ({(12/25)*sqrt(6)/6},{(20/25)*(sqrt(2)/2)})--({(11/25)*(sqrt(6)/6)},{(21/25)*(sqrt(2)/2)});
\draw ({(11/25)*sqrt(6)/6},{(21/25)*(sqrt(2)/2)})--({(25/25)*(sqrt(6)/6)},{(25/25)*(sqrt(2)/2)});

\draw ({(-7/25)*sqrt(6)/6},{(19/25)*(sqrt(2)/2)})--({(-11/25)*(sqrt(6)/6)},{(21/25)*(sqrt(2)/2)});
\draw ({(-7/25)*sqrt(6)/6},{(19/25)*(sqrt(2)/2)})--(0,{(20/25)*(sqrt(2)/2)});
\draw ({(-11/25)*sqrt(6)/6},{(21/25)*(sqrt(2)/2)})--(0,{(20/25)*(sqrt(2)/2)});

\draw ({(7/25)*sqrt(6)/6},{(19/25)*(sqrt(2)/2)})--({(11/25)*(sqrt(6)/6)},{(21/25)*(sqrt(2)/2)});
\draw ({(7/25)*sqrt(6)/6},{(19/25)*(sqrt(2)/2)})--(0,{(20/25)*(sqrt(2)/2)});
\draw ({(11/25)*sqrt(6)/6},{(21/25)*(sqrt(2)/2)})--(0,{(20/25)*(sqrt(2)/2)});

\end{tikzpicture}
\caption{A barycentric embedding of the second graph approximation of $SG_2$.}
\end{figure}
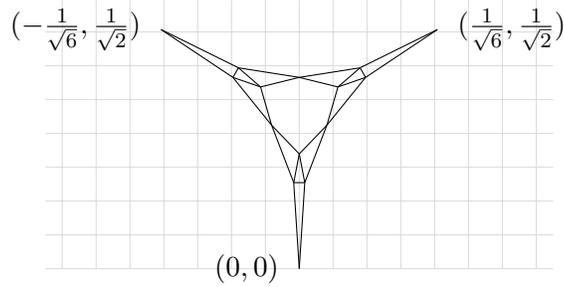

A generalization of this is the so-called harmonic Sierpi\'nski gasket. We refer the reader to \cite{teplyaev2005harmonic,teplyaev2004energy}. Our theorem then proves that we can construct the harmonic $SG_k$ in a non-degenerate way for all $k \geq 2$.

\section{Radon--Nikodym derivatives}
As we have seen from \cite{hino2009energy}, combined with theorem 1.1, that all energy measures $\nu_{h_1}$ and $\nu_{h_2}$ are mutually absolutely continuous with respect to each other on $SG_k$ and therefore the Radon--Nikodym derivatives $\frac{d\nu_{h_1}}{d\nu_{h_2}}$ exist. For $SG_2$, there is a conjecture of Strichartz and Tse in \cite{strichartz2010local} that they belong in $L^p(d\nu_{h_2})$ for $p<\frac{\log{15}}{\log{9}}$ and numerical evidence was presented there. It was shown that this is equivalent to proving that the sum
$$S(m,p)=\sum_{|w|=m} \nu_{h_1}(F_wK)^p \nu_{h_2}(F_wK)^{1-p}$$
is uniformly bounded as $m \to \infty$. It is also interesting to know whether the equivalent statement is true for all $SG_k$, in other words whether the Radon--Nikodym derivatives belong in $L^{p_k}(\nu_{h_2})$ for appropriate values of $p_k$. A restriction, as shown in \cite{strichartz2010local} for $SG_2$, is that the value $p_k$ cannot be much larger than $1$. The reasoning is that combining a symmetric function and a skew-symmetric function makes an individual term in the sum unbounded for larger $p$. We need not use necessarily only a symmetric function, as it is actually easy to compute the general decay of all energy measures in the direction of all three boundary points. If we take a normalized harmonic function on $SG_2$ with boundary values $(h(q_1),h(q_2),h(q_3)=(0,a,1)$ then after some computations we can find that
\begin{equation*}
\begin{split}
\nu_h(F_1^mK)=&\frac{(a+1)^2}{2}\left(\frac{3}{5}\right)^m+\frac{3}{2}(a-1)^2\frac{1}{15^m}\\
\nu_h(F_2^mK)=&2(a-\frac{1}{2})\left(\frac{3}{5}\right)^m+\frac{3}{2}\frac{1}{15^m}\\
\nu_h(F_3^mK)=&\frac{(a-2)^2}{2}\left(\frac{3}{5}\right)^m+\frac{3}{2}a^2\frac{1}{15^m}.\\
\end{split}
\end{equation*}
Note that the terms $3/5$ and $1/15$ correspond to the second and third eigenvalue of the harmonic extension matrices of $SG_2$. For general $SG_k$ the equivalent limitation is that $p_k<\frac{\log{(1/\lambda_k)}}{\log{(r_k/\lambda_k)}}$ where $\lambda_k$ is the smallest eigenvalue in the harmonic extension matrix corresponding to one of the boundary vertices. Indeed, as before, taking a symmetric and skew-symmetric harmonic function with respect to $q_1$ will make the sum unbounded since the individual term corresponding to $F_1^mK$ will have value
$$\nu_{h_1}(F_1^mK)^{p_k}\nu_{h_2}(F_1^mK)^{1-p_k}=c(r_k^m)^{p_k}(\lambda_k^m)^{1-p_k}=[(\frac{r_k}{\lambda_k})^{p_k}\lambda_k]^m \to \infty$$
for larger values of $p_k$. In \cite{strichartz2010local} there were some numerical results in favor of the conjecture for $SG_2$. The main idea is to use a ratio test having a value less than $1$. We present similar evidence for $SG_3$ for $p_3<\frac{\log{105}}{\log{49}}$. To calculate the measure on the cells, we use the methodology of section 6 in \cite{azzam2008conformal}. We note a typo in the $E_0$ and $E_3$ matrices. It should be 
$$E_0 = \frac{1}{7875}\left(
\begin{array}{ccc}
3701&-49&-49\\
962&287&-238\\
962&-238&287
\end{array}
\right) \, \text{ and } \,
E_3 = \frac{1}{31500}\left(
\begin{array}{ccc}
1174&49&49\\
-962&3613&1213\\
-962&1213&3613
\end{array}
\right)$$
which give the corresponding value of $dim_2\nu_h=\frac{1031}{3675}$ in Theorem 6.1. We keep the notation from \cite{strichartz2010local} where $R(m,p)=\frac{S(m,p)-S(m-1,p)}{S(m-1,p)-S(m-2,p)}$.  The harmonic functions we use in the data in Figure 6 are $h_1=(0,1,1)$ and $h_2=(0,1,-1)$. We refer the reader to \cite{strichartz2010local} for more details regarding the case of $SG_2$.

\begin{figure}
\begin{tabular}{ |p{0.45cm}| |p{1.4cm}||p{1.5cm}|p{1.66cm}|  }
 \hline
 \hline
 $m$ & $p=1.1$& $p=1.14$ & $p=1.185$\\
 \hline
  3& \hspace{1mm} 0.7865 & \hspace{1mm}  0.8940  &   \hspace{1mm} 1.0392   \\
  4& \hspace{1mm} 0.7511 & \hspace{1mm} 0.8598  &    \hspace{1mm} 1.0068    \\
  5& \hspace{1mm} 0.7424 &  \hspace{1mm} 0.8514   &   \hspace{1mm}   0.9991  \\
  6& \hspace{1mm} 0.7345 &  \hspace{1mm} 0.8447  &    \hspace{1mm}   0.9938  \\
  7& \hspace{1mm} 0.7297 &  \hspace{1mm} 0.8409  &   \hspace{1mm}   0.9910    \\
  8& \hspace{1mm} 0.7269 & \hspace{1mm}  0.8389  &     \hspace{1mm}  0.9897     \\
 \hline
\end{tabular}
\caption{R(m,p) values for $SG_3$. We omit higher $m$ because we cannot maintain both precision and feasible computational times.}
\end{figure}

\section*{Acknowledgments}
The author thanks Anders Karlsson, Robert S. Strichartz, Alexader Teplyaev and Anders \"Oberg  for valuable suggestions that contributed to the improvement of this article. The author is also especially grateful to Anders Johansson for helpful discussions and suggesting the graph geometric approach.

\nocite{*}
\bibliographystyle{abbrv}
\bibliography{Biblio}

\end{document}